\documentclass[10pt,reqno]{amsart}

\pdftrailerid{}
\usepackage[OT1]{fontenc}
\usepackage[margin=1.08in]{geometry}
\usepackage{microtype}
\usepackage{amsmath,amssymb,mathtools,mathrsfs}
\usepackage{enumitem}
\usepackage[hidelinks]{hyperref}
\hypersetup{
  pdftitle={Exceptional Sets of Positive Capacity in Gross's Star Theorem},
  pdfauthor={Sina Nadi},
  pdfsubject={Exceptional Sets in Gross's Star Theorem},
  pdfcreator={},
  pdfproducer={},
  pdfkeywords={},
}

\setlist[enumerate]{label=(\roman*),leftmargin=2.15em,itemsep=0.15em,topsep=0.35em}
\setlist[itemize]{leftmargin=1.8em,itemsep=0.15em,topsep=0.35em}
\numberwithin{equation}{section}

\newtheorem{theorem}{Theorem}[section]
\newtheorem{lemma}[theorem]{Lemma}
\newtheorem{proposition}[theorem]{Proposition}

\theoremstyle{remark}
\newtheorem{remark}[theorem]{Remark}

\newcommand{\C}{\mathbb C}
\newcommand{\D}{\mathbb D}
\newcommand{\Chat}{\widehat{\mathbb C}}
\newcommand{\T}{\mathbb T}
\newcommand{\eps}{\varepsilon}
\newcommand{\dd}{\,\mathrm d}
\newcommand{\dist}{\operatorname{dist}}
\newcommand{\diam}{\operatorname{diam}}
\newcommand{\caplog}{\operatorname{cap}}
\newcommand{\Mod}{\operatorname{Mod}_2}
\newcommand{\Real}[1]{\operatorname{\mathit{Re}}\left\{#1\right\}}

\newcommand{\Log}{\operatorname{Log}}
\newcommand{\lemref}[1]{\hyperref[#1]{Lemma~\ref*{#1}}}
\newcommand{\propref}[1]{\hyperref[#1]{Proposition~\ref*{#1}}}
\newcommand{\thmref}[1]{\hyperref[#1]{Theorem~\ref*{#1}}}
\newcommand{\corref}[1]{\hyperref[#1]{Corollary~\ref*{#1}}}
\newcommand{\secref}[1]{\hyperref[#1]{Section~\ref*{#1}}}

\makeatletter
\def\@secnumfont{\bfseries}
\def\section{\@startsection{section}{1}%
  \z@{.7\linespacing\@plus\linespacing}{.5\linespacing}%
  {\normalfont\bfseries\centering}}
\makeatother

\title[Exceptional Sets of Positive Capacity]
{Exceptional Sets of Positive Capacity in Gross's Star Theorem}
\author[Sina Nadi]{Sina Nadi}
\date{}

\begin{document}

\begin{abstract}
In 1918, Gross proved that a regular local inverse of a meromorphic
function in the plane can be continued analytically along every ray from
its centre except for directions in a set of Lebesgue measure zero.
In 1977, Nagasaka asked whether this exceptional set must have logarithmic
capacity zero. We prove that there exist a transcendental entire function
and a regular local inverse for which the exceptional set contains a
compact set of positive logarithmic capacity. In fact, for every
$0<s<1$, such a function and inverse can be chosen so that this compact set
has positive $s$-dimensional Hausdorff measure.
\end{abstract}

\maketitle

\section{Introduction}\label{sec:introduction}

For $z\in\C$ and $r>0$, we write
$\Delta(z,r)=\{w\in\C:|w-z|<r\}$, $\D=\Delta(0,1)$, and
$\T=\partial\D$.
Let $f$ be meromorphic in $\C$. Suppose that $f(z_0)\in\C$ and
$f'(z_0)\ne0$. Put $w_0=f(z_0)$ and denote by $\phi$ the inverse germ
determined by
\[
 \phi(w_0)=z_0,
 \qquad
 f(\phi(w))=w.
\]

In 1918, Gross proved that, for almost every $\theta$ with respect to
Lebesgue measure, $\phi$ admits analytic continuation along the ray
\begin{equation}\label{eq:gross-ray}
 w=w_0+te^{i\theta},
 \qquad 0\leq t<+\infty.
\end{equation}
See \cite{Gross1918}.
The exceptional set $E(f,\phi)$ consists of the points $e^{i\theta}\in\T$
for which $\phi$ does not admit analytic continuation along the whole ray
\eqref{eq:gross-ray}.
Although the exceptional set has measure zero, it need not be countable.
Volkovyskii constructed a meromorphic function in the plane with
uncountably many exceptional directions
\cite[Theorem~17]{Volkovyskii1950}. The exceptional set in this
example has logarithmic capacity zero \cite[Section~4, Question~1]{Eremenko2026}, so the example
leaves open the possibility that capacity imposes a stronger restriction
than Lebesgue measure.

Kuramochi asked whether the exceptional set must satisfy a stronger
restriction when the covering surface is conformally equivalent to
$\C$ \cite[p.~85]{Kuramochi1963}. For the larger class of covering
surfaces without a Green function, Nagasaka attributed the question of
vanishing logarithmic capacity to Noshiro \cite[p.~1]{Nagasaka1977}.
He proved that every $F_\sigma$ set of Lebesgue measure zero occurs as the
exceptional set of an infinitely connected planar covering surface without
a Green function. These surfaces are not conformally equivalent to $\C$, and
the plane case therefore remained open. In 1977, Nagasaka stated the
following problem for that case \cite[p.~2]{Nagasaka1977}.

\medskip
\noindent\textbf{Nagasaka's problem.}
Let $f$ be a nonconstant meromorphic function in $\C$, and let $\phi$ be a
regular local inverse of $f$ at $w_0\in\C$. Must $E(f,\phi)$ have
logarithmic capacity zero?
\medskip

A related problem was posed by Eremenko and recorded as
Problem~1.32 in the 1984 collection of Barth, Brannan, and Hayman
\cite{BarthBrannanHayman1984}. That question asks whether Gross's
exceptional set satisfies a stronger restriction. Stephenson explicitly
formulated the alternatives of Hausdorff dimension zero and logarithmic
capacity zero, and referred to this problem
\cite[p.~337]{Stephenson1988}. Hayman and Lingham reproduce the 1984
question with its attribution to Eremenko
\cite[Problem~1.32]{HaymanLingham2019}. Eremenko also states the capacity
question \cite{Eremenko2026}.

We prove that a transcendental entire function can have an inverse germ
whose exceptional set has positive logarithmic capacity. Here and below,
$\caplog$ denotes logarithmic capacity, and $\mathcal H^s$ denotes
$s$-dimensional Hausdorff measure.

\begin{theorem}\label{thm:main}
For every $0<s<1$, there exist a transcendental entire function $f$, a point $z_0$ with
$f'(z_0)\ne0$, and the corresponding inverse germ $\phi$ at
$w_0=f(z_0)$ such that
\[
 \caplog E(f,\phi)>0.
\]

More precisely, after an affine change in the range one may take $w_0=0$.
There is a compact set $K\subset\mathbb R$ for which
\[
 \caplog\{e^{i\theta}:\theta\in K\}>0,
 \qquad
 \mathcal H^s\bigl(\{e^{i\theta}:\theta\in K\}\bigr)>0,
\]
and, for every $\theta\in K$, the germ $\phi$ admits analytic continuation
along $[0,e^{i\theta})$ but not through $e^{i\theta}$.
\end{theorem}

Our construction also addresses an earlier question of Hayman, recorded by
Ohtsuka in 1954, which asks whether, for a prescribed
$0<\alpha<1$, there is an entire function whose inverse germ has an
exceptional set of positive $\alpha$-dimensional Hausdorff measure
\cite[Section~1, Remark~3]{Ohtsuka1954}. \thmref{thm:main} gives such an example.

In \secref{sec:bounded-domain}, we construct a bounded starlike domain and
establish the capacity estimates for its boundary correspondence. This
domain is used in \secref{sec:conformal-map} to construct the conformal map
and the closed set needed for approximation by entire functions. We carry
out this approximation in \secref{sec:entire-function} and prove
\thmref{thm:main}.

\section{Construction of the domain}\label{sec:bounded-domain}

Fix $0<s<1$. We first construct a compact set $K$ of angles. Choose
\begin{equation}\label{eq:constants}
 a=2^{-1/s},
 \qquad
 2a<b<1,
 \qquad
 \frac ab<q<\frac12,
 \qquad
 \frac12<\tau<\min\left\{\frac34,\frac1{4a}\right\}.
\end{equation}
These choices are possible since $a<1/2$. For $n=0,1,2,\ldots$, put
\[
 \ell_0=\frac14,
 \qquad
 \ell_n=\ell_0a^n.
\]
All constants below may depend on these fixed parameters.

A binary word is a finite sequence with entries in $\{0,1\}$.
We denote its length by $|v|$ and the empty word by $\varnothing$.
Let $J_\varnothing$ be the closed interval of length $\ell_0$ centred at
zero. If $J_v$ has length $\ell_n$ and centre $c_v$, define $J_{v0}$ and
$J_{v1}$ to have length $a\ell_n$ and centres
\[
 c_v-\frac{\ell_n}{4}
 \quad\hbox{and}\quad
 c_v+\frac{\ell_n}{4},
\]
respectively. Set
\begin{equation}\label{eq:Iv-Un-K}
 I_v=(c_v-\tau\ell_n,c_v+\tau\ell_n),
 \qquad
 U_n=\bigcup_{|v|=n}I_v,
 \qquad
 K=\bigcap_{n\geq0}\bigcup_{|v|=n}J_v.
\end{equation}

The intervals $J_v$ with $|v|=n$ are called the intervals of level $n$.
For $\theta\in K$, let $v_n(\theta)$ be the unique word $v\in\{0,1\}^n$
for which $\theta\in J_v$. The intervals $I_v$ at any fixed level are
mutually disjoint. Moreover,
\begin{equation}\label{eq:interval-separation}
 \overline{I_{vj}}\subset\operatorname{int}J_v
 \quad (j=0,1),
 \qquad
 \dist(\theta,\partial I_v)\geq\left(\tau-\frac12\right)\ell_n
\end{equation}
whenever $\theta\in K\cap J_v$.

For a nonempty compact set $X\subset\C$, let $\mathcal P(X)$ denote the
probability measures supported on $X$. The logarithmic energy of
$\nu\in\mathcal P(X)$ and the logarithmic capacity of $X$ are
\[
 I(\nu)=\iint\log\frac1{|x-y|}\dd\nu(x)\dd\nu(y),
 \qquad
 \caplog X=\exp\left(-\inf_{\nu\in\mathcal P(X)}I(\nu)\right).
\]

Thus $\caplog X>0$ if and only if $X$ supports a probability measure of
finite logarithmic energy. For arbitrary sets, we use outer logarithmic
capacity.

\begin{lemma}\label{lem:positive-capacity}
The compact set
\[
 E=\{e^{i\theta}:\theta\in K\}
\]

has positive $s$-dimensional Hausdorff measure and positive logarithmic
capacity.
\end{lemma}

\begin{proof}
Let $\nu$ be the probability measure on $K$ determined by
$\nu(K\cap J_v)=2^{-|v|}$. If $n\geq1$ and an interval
$L\subset\mathbb R$ has length $\ell_n\leq|L|<\ell_{n-1}$, then $L$
meets at most two intervals of level $n-1$, since meeting three would
require it to contain the middle interval, whose length is $\ell_{n-1}$.
Thus $L$ meets at most four intervals of level $n$.
Consequently,
\begin{equation}\label{eq:frostman-K}
 \nu(L)\leq 4\cdot2^{-n}
 =4\left(\frac{\ell_n}{\ell_0}\right)^s
 \leq C_s|L|^s.
\end{equation}

Here one may take $C_s=4\ell_0^{-s}$. For $|L|\geq\ell_0$, the same
upper bound follows from $\nu(L)\leq1$.
For every cover of $K$ by intervals $L_j$,
\[
 1=\nu(K)\leq\sum_j\nu(L_j)
 \leq C_s\sum_j|L_j|^s.
\]

Hence $K$ has positive $s$-dimensional Hausdorff measure. The
exponential map is bi-Lipschitz on $J_\varnothing$, so the same is true
of $E$.

For $\theta,\eta\in K$, we have $|\theta-\eta|\leq\ell_0<1$. Thus,
for each $\theta\in K$, Tonelli's theorem and \eqref{eq:frostman-K} give
\[
 \int\log\frac1{|\theta-\eta|}\dd\nu(\eta)
 =\int_0^\infty
   \nu\bigl(\{\eta:|\theta-\eta|<e^{-u}\}\bigr)\dd u
 \leq 2^sC_s\int_0^\infty e^{-su}\dd u<+\infty.
\]

Hence $I(\nu)<+\infty$. Define a probability measure $\eta$ on $E$ by
$\eta(M)=\nu(\{\theta\in K:e^{i\theta}\in M\})$ for every Borel set
$M\subseteq E$. Since the exponential map is bi-Lipschitz on
$J_\varnothing$, the measure $\eta$ also has finite logarithmic energy.
It follows that $\caplog E>0$.
\end{proof}

For $n=0,1,2,\ldots$, define
\begin{equation}\label{eq:radii}
 r_n=1-\frac{b^n}{2},
 \qquad
 h_n=r_{n+1}-r_n=\frac{(1-b)b^n}{2}.
\end{equation}

Consider the bounded starlike domain
\begin{equation}\label{eq:D}
 D=\Delta(0,r_0)\cup
 \bigcup_{n\geq0}
 \{re^{i\theta}:0<r<r_{n+1},\ \theta\in U_n\}.
\end{equation}

Each set in \eqref{eq:D} is open, so $D$ is open. Also, $[0,z]\subset D$
for every $z\in D$, so $D$ is starlike with respect to the origin. The nesting
in \eqref{eq:interval-separation} gives
\[
 \{\theta\in(-\pi,\pi):re^{i\theta}\in D\}=U_n
 \quad\hbox{when }r_n\leq r<r_{n+1}.
\]

In particular,
\begin{equation}\label{eq:radial-segments}
 [0,e^{i\theta})\subset D,
 \qquad
 e^{i\theta}\in\partial D
 \quad (\theta\in K).
\end{equation}

By a crosscut of a domain, we mean a simple open arc in the domain whose
closure is a closed arc with two distinct endpoints on its boundary.
For $|v|=n$, put
\begin{equation}\label{eq:crosscuts}
 x_n=r_n+\frac{h_n}{4},
 \qquad
 y_n=r_n+\frac{3h_n}{4},
 \qquad
 \gamma_v=\{y_ne^{i\theta}:\theta\in I_v\}.
\end{equation}

The arc $\gamma_v$ is a crosscut of $D$. Let $D_v^+$ denote the component
of $D\setminus\gamma_v$ that does not contain the origin.

For a family $\Gamma$ of locally rectifiable curves, its $2$-modulus and
extremal length are
\[
 \Mod\Gamma=\inf_\rho\int_\C\rho^2\dd A,
 \qquad
 \lambda(\Gamma)=\frac1{\Mod\Gamma}.
\]

Here $dA$ denotes planar area measure, and the infimum is taken over
nonnegative Borel functions $\rho$ such that $\int_\gamma\rho\,ds\geq1$
for every $\gamma\in\Gamma$. For $A_1,A_2\subseteq G$, let
$\Gamma_G(A_1,A_2)$ be the family of locally rectifiable curves in $G$
that join $A_1$ to $A_2$. We write
\[
 \operatorname{Mod}_{2,G}(A_1,A_2)=\Mod\Gamma_G(A_1,A_2),
 \qquad
 \lambda_G(A_1,A_2)=\frac1{\operatorname{Mod}_{2,G}(A_1,A_2)},
\]
where $\lambda_G(A_1,A_2)$ is the extremal distance between $A_1$ and $A_2$
in $G$.

We now estimate the extremal distance from a fixed disc about the origin
to the crosscuts $\gamma_v$. Fix $0<r_*<r_0$ and put
$C_{\mathrm{cen}}=\overline{\Delta(0,r_*)}$. Every curve from
$C_{\mathrm{cen}}$ to
$\gamma_v$ crosses the quadrilateral
\[
 \{z:x_n<|z|<y_n,\ \arg z\in I_v\}
\]
from its side on $|z|=x_n$ to its side on $|z|=y_n$. Applying
$z\mapsto\log z$ and the formula for the extremal distance between opposite
sides of a rectangle, we obtain
\begin{equation}\label{eq:large-extremal-distance}
 \lambda_D(C_{\mathrm{cen}},\gamma_v)
 \geq\frac{\log(y_n/x_n)}{|I_v|}
 \geq\frac{h_n}{4\tau\ell_0a^n}
 =\frac{1-b}{8\tau\ell_0}\left(\frac ba\right)^n.
\end{equation}

We shall use the following estimate to bound the diameter of the component
of $\D$ that lies beyond each crosscut under a Riemann map.

\begin{lemma}\label{lem:component-diameter}
Fix $0<\rho<1$. Let $\gamma$ be a crosscut of $\D$ that is disjoint from
$\overline{\Delta(0,\rho)}$, and let $V$ be the component of
$\D\setminus\gamma$ that does not contain this disc. There are constants
$c_\rho,C_\rho>0$ such that
\begin{equation}\label{eq:component-diameter}
 \diam V\leq C_\rho
 \exp\{-c_\rho\lambda_\D(\overline{\Delta(0,\rho)},\gamma)\}.
\end{equation}
\end{lemma}

\begin{proof}
Write $d=\diam V$. Choose two points of $V$ at distance greater than
$d/2$ and join them by a compact arc $F\subset V$. Every curve from
$\overline{\Delta(0,\rho)}$ to $F$ contains a subcurve from
$\overline{\Delta(0,\rho)}$ to $\gamma$. It follows that
\begin{equation}\label{eq:extremal-monotonicity}
 \lambda_\D(\overline{\Delta(0,\rho)},\gamma)
 \leq
 \lambda_\D(\overline{\Delta(0,\rho)},F).
\end{equation}

The modulus of the curves that join $\overline{\Delta(0,\rho)}$ to $F$
is the infimum of the Dirichlet energies
$\int_\D|\nabla u|^2\dd A$ over continuous functions
$u\in W^{1,2}(\D)$ that satisfy $0\leq u\leq1$, $u=0$ on
$\overline{\Delta(0,\rho)}$, and $u=1$ on $F$
\cite[Section~2.D]{BuckleyHerron2007}. Here $W^{1,2}(\D)$ is the Sobolev
space of square-integrable functions with square-integrable first weak
derivatives.

The same infimum is obtained by functions that are smooth on a neighborhood
of $\overline\D$. To verify this, for $0<\varepsilon<1/2$ set
\[
 u_\varepsilon=\min\left\{1,\max\left\{0,
 \frac{u-\varepsilon}{1-2\varepsilon}\right\}\right\}.
\]
This function is constant in neighborhoods of the two compact sets, and
its energy is at most $(1-2\varepsilon)^{-2}$ times that of $u$.
For $0<r<1$ sufficiently close to one, the function
$u_{\varepsilon,r}(z)=u_\varepsilon(rz)$ preserves these constant values
and has no larger energy on $\D$. It is defined on a neighborhood of
$\overline\D$. Convolution with a nonnegative smooth kernel of integral
one and sufficiently small support preserves the constant values. As the
support shrinks to $\{0\}$, the resulting energies on $\D$ converge to
$\int_\D|\nabla u_{\varepsilon,r}|^2\dd A$. Letting $r\to1$ and
$\varepsilon\to0$ proves the equality of the two infima.

For such a smooth potential, define
\[
 U(z)=
 \begin{cases}
 u(z),& |z|\leq1,\\
 u(1/\overline z),& |z|>1.
 \end{cases}
\]

This function is continuous, has square-integrable weak derivatives, and
has the prescribed values on both compact sets. Inversion in $\T$
preserves Dirichlet energy, so
\[
 \int_\C|\nabla U|^2\dd A=2\int_\D|\nabla u|^2\dd A.
\]

Taking the infimum over admissible potentials gives
\begin{equation}\label{eq:reflection-modulus}
 \operatorname{Mod}_{2,\D}(\overline{\Delta(0,\rho)},F)
 \geq\frac12
 \operatorname{Mod}_{2,\C}(\overline{\Delta(0,\rho)},F).
\end{equation}

By a continuum, we mean a nonempty compact connected set. It is nondegenerate
if it contains more than one point. A modulus estimate for two plane
continua
\cite[Lemma~2.2(b) and (e)]{BuckleyHerron2007} gives
\begin{equation}\label{eq:plane-continuum-modulus}
 \operatorname{Mod}_{2,\C}(A_1,A_2)
 \geq\frac{c}{\log(2+\Delta(A_1,A_2))},
 \qquad
 \Delta(A_1,A_2)=
 \frac{\dist(A_1,A_2)}
 {\min\{\diam A_1,\diam A_2\}}
\end{equation}
for disjoint nondegenerate plane continua $A_1$ and $A_2$, with an
absolute constant $c>0$. In our application the first continuum is fixed,
$\diam F>d/2$, and both continua lie in $\overline\D$. Equations
\eqref{eq:reflection-modulus} and \eqref{eq:plane-continuum-modulus} imply
\[
 \lambda_\D(\overline{\Delta(0,\rho)},F)
 \leq C_\rho\log\frac{C_\rho}{d}.
\]

Combining this inequality with \eqref{eq:extremal-monotonicity} and
exponentiating proves \eqref{eq:component-diameter}, with positive constants that
depend only on $\rho$.
\end{proof}

Since $D$ is starlike, it is simply connected. Let
$\chi\colon\D\to D$ be conformal with $\chi(0)=0$. By Carath\'eodory's
prime end theorem, $\chi$ identifies $\T$ with the prime end boundary of
$D$. Choose $\rho>0$
so that
\[
 \chi(\overline{\Delta(0,\rho)})\subset C_{\mathrm{cen}}.
\]

Each endpoint of $\gamma_v$ lies in the interior of a straight radial
segment of $\partial D$. Schwarz reflection extends $\chi^{-1}$
conformally across these segments, so $\chi^{-1}(\gamma_v)$ is a crosscut of $\D$.
Let $V_v$ be the component of $\D\setminus\chi^{-1}(\gamma_v)$ that does
not contain $\overline{\Delta(0,\rho)}$. Conformal invariance of extremal
distance,
\eqref{eq:large-extremal-distance}, and \lemref{lem:component-diameter} give
\begin{equation}\label{eq:delta-n}
 \diam\overline{V_v}\leq\delta_n,
 \qquad
 \delta_n=C\exp\left\{-c\left(\frac ba\right)^n\right\},
 \qquad |v|=n.
\end{equation}

For $\theta\in K$, the nonempty compact sets
$\overline{V_{v_n(\theta)}}$ are nested and their diameters tend to zero.
Their intersection therefore consists of one point, which we denote by
$\beta(\theta)$. Since
$\chi(V_{v_n(\theta)})=D_{v_n(\theta)}^+$ lies outside the circle of radius
$y_n$, an interior point $\beta(\theta)\in\D$ would satisfy
$|\chi(\beta(\theta))|\geq y_n$ for every $n$. As $y_n\to1$, this is
impossible. Thus $\beta(\theta)\in\T$. In $D$ one
has
\[
 \diam D_v^+\leq 1-y_n+|I_v|\leq C(b^n+a^n).
\]

The crosscuts $\gamma_{v_n(\theta)}$ have pairwise disjoint closures,
are nested towards $e^{i\theta}$, and have diameters tending to zero.
They therefore represent a prime end whose impression,
that is, the intersection $\bigcap_n\overline{D_{v_n(\theta)}^+}$, is
$\{e^{i\theta}\}$. Distinct values of $\theta$ give distinct prime ends.
Hence $\beta$ is injective. If two parameters in $K$ belong to the same
interval $J_v$ of level $n$, their images under $\beta$ lie in the same
set $\overline{V_v}$, whose diameter is at most $\delta_n$. Formula
\eqref{eq:delta-n} therefore also shows that $\beta$ is continuous. Put
\begin{equation}\label{eq:B-Bv}
 B=\beta(K),
 \qquad
 B_v=\beta(K\cap J_v).
\end{equation}

The sets $B_v$ form a disjoint level $n$ partition of $B$ and
\begin{equation}\label{eq:Bv-diameter}
 \diam B_v\leq\delta_n.
\end{equation}

Thus the intervals $J_v$ have length $\ell_0a^n$, whereas the corresponding
sets $B_v\subset\T$ have diameter at most $C\exp\{-c(b/a)^n\}$.

\begin{proposition}\label{prop:B-zero-capacity}
The compact set $B$ has logarithmic capacity zero.
\end{proposition}

\begin{proof}
Let $\sigma$ be a probability measure on $B$ and put
$m_v=\sigma(B_v)$. Since $|x-y|\leq2$ on $\T$, the kernel
$\log(2/|x-y|)$ is nonnegative. For every sufficiently large $n$, one has
$\delta_n<2$, and
\begin{align*}
 I(\sigma)+\log2
 &=\iint\log\frac2{|x-y|}\dd\sigma(x)\dd\sigma(y)\\
 &\geq\log\frac2{\delta_n}\sum_{|v|=n}m_v^2\\
 &\geq2^{-n}\log\frac2{\delta_n}.
\end{align*}

By \eqref{eq:delta-n}, the last expression is bounded below by
$c(b/(2a))^n-O(2^{-n})$ and tends to $+\infty$ because $b>2a$. Every
probability measure on $B$ has infinite energy. Thus $\caplog B=0$.
\end{proof}

We finish this section by estimating the radii of discs centered on the
radial segments $[0,e^{i\theta})$, where $\theta\in K$.

\begin{lemma}\label{lem:boundary-distance}
Write $d_D(z)=\dist(z,\partial D)$. There are constants $c_0,C_0>0$
such that
\begin{equation}\label{eq:boundary-distance}
 c_0a^{n+1}\leq d_D(te^{i\theta})\leq C_0a^n
\end{equation}
whenever $\theta\in K$ and $r_n\leq t<r_{n+1}$.
\end{lemma}

\begin{proof}
Let $\alpha$ be an endpoint of $I_{v_n(\theta)}$. Then
$te^{i\alpha}\in\partial D$ and
$|te^{i\theta}-te^{i\alpha}|\leq t|\theta-\alpha|\leq C\ell_n$.
This proves the upper estimate.

For the lower estimate choose $\kappa>0$ so small that, for every $n$,
\[
 4\kappa a^{n+1}<\frac12\left(\tau-\frac12\right)\ell_{n+1},
 \qquad
 \kappa a^{n+1}<\tfrac14\min\{h_n,h_{n+1}\}.
\]

Such a choice is possible because $a/b<1$. Since $t\geq r_0=1/2$,
every $z\in\Delta(te^{i\theta},\kappa a^{n+1})$ satisfies
\[
 \lvert\arg z-\theta\rvert<4\kappa a^{n+1}
 <\frac12\left(\tau-\frac12\right)\ell_{n+1},
 \qquad
 \bigl||z|-t\bigr|<\kappa a^{n+1}.
\]

Here $\arg z\in(-\pi,\pi)$. The second inequality implies that either
$|z|<r_0$ or $r_k\leq|z|<r_{k+1}$ for some $k\geq0$ with $|k-n|\leq1$.
By \eqref{eq:interval-separation}, the first inequality places $\arg z$ in
$I_{v_{n+1}(\theta)}$. The nesting of the intervals then places it in
$I_{v_k(\theta)}$ whenever $|k-n|\leq1$. Thus every point of
$\Delta(te^{i\theta},\kappa a^{n+1})$ belongs to $D$, which proves the
lower estimate with $c_0=\kappa$.
\end{proof}

\section{Conformal mapping and an approximation set}\label{sec:conformal-map}

We now map $D$ conformally onto an unbounded starlike domain using a
construction similar to Mori's \cite[Theorem~2]{Mori1950}; see also
Ohtsuka \cite[Theorem~10]{Ohtsuka1954}.

The choice of measure below provides the estimates needed for
approximation by entire functions while preserving the continuations of
one inverse germ along all the prescribed radial segments.

For every word $v$ of length $n\geq1$, choose a point $\beta_v\in B_v$ and
define
\begin{equation}\label{eq:mu}
 \mu=c_q\sum_{n\geq1}q^n\sum_{|v|=n}\delta_{\beta_v},
 \qquad c_q=\frac{1-2q}{2q}.
\end{equation}

Here $\delta_\beta$ denotes the unit point mass at $\beta$. The measure
$\mu$ is a probability measure because
\[
 \|\mu\|=c_q\sum_{n\geq1}(2q)^n=1.
\]

The chosen points are dense in $B$, so $\operatorname{supp}\mu=B$. On
$\D$, define
\begin{equation}\label{eq:A}
 A(\zeta)=-\int_B\Log(1-\overline\xi\zeta)\dd\mu(\xi),
\end{equation}
where the branch is normalized by $\Log1=0$. The integral converges locally
uniformly and defines a holomorphic function. Notice that
$-\log|1-\overline\xi\zeta|>-\log2$ for $\xi\in B$ and $\zeta\in\D$,
since $|1-\overline\xi\zeta|<2$.

If $\zeta\in V_v$ and $|v|=n\geq1$, then
$|\zeta-\beta_v|\leq\delta_n$. We separate the summand
$c_qq^n\delta_{\beta_v}$ in \eqref{eq:mu} and use the preceding lower bound
to obtain
\begin{equation}\label{eq:A-on-Vv}
 \Real{A(\zeta)}
 \geq c_qq^n\log\frac1{\delta_n}-\log2
 \geq c_1\left(\frac{qb}{a}\right)^n-C_1.
\end{equation}

To obtain uniform escape to infinity under the conformal map below, we
also need a lower bound valid for every $\zeta\in\D$ close to $B$.
Suppose that $n\geq1$, $\zeta\in\D$, and $\dist(\zeta,B)<\delta_n$.
Choose $\beta\in B$ with
$|\zeta-\beta|<\delta_n$. Let $v$ be the unique word of length $n$ for
which $\beta\in B_v$. Then $|\beta-\beta_v|\leq\delta_n$. Hence
\begin{equation}\label{eq:A-unrestricted}
 \Real{A(\zeta)}
 \geq c_qq^n\log\frac1{2\delta_n}-\log2.
\end{equation}

Since $\delta_n=C\exp\{-c(b/a)^n\}$, the right-hand side is at least
$c(qb/a)^n-O(q^n)-\log2$ and tends to $+\infty$ because $qb>a$. We have
proved
\begin{equation}\label{eq:A-blowup}
 \Real{A(\zeta)}\longrightarrow+\infty
 \quad\hbox{uniformly as}\quad
 \dist(\zeta,B)\longrightarrow0,
 \qquad \zeta\in\D.
\end{equation}

Set
\begin{equation}\label{eq:Q}
 Q(\zeta)=\zeta\exp(2A(\zeta)).
\end{equation}

Since $\mu(B)=1$,
\begin{equation}\label{eq:starlike-criterion}
 \frac{\zeta Q'(\zeta)}{Q(\zeta)}
 =\int_B\frac{1+\overline\xi\zeta}
                 {1-\overline\xi\zeta}\dd\mu(\xi).
\end{equation}

The real part of the integrand equals
\[
 \frac{1-|\zeta|^2}{|1-\overline\xi\zeta|^2}>0.
\]

Moreover, $Q(0)=0$ and $Q'(0)=1$, and the quotient in
\eqref{eq:starlike-criterion} has its removable value $1$ at zero.
The analytic characterization of starlike functions states that a
holomorphic function $h$ with $h(0)=0$, $h'(0)\ne0$, and
$\Real{\zeta h'(\zeta)/h(\zeta)}>0$ on $\D$ is univalent and has an image
that is starlike with respect to the origin. It follows that $Q$ is univalent
and $\Omega=Q(\D)$ is starlike with respect to the origin. Since
$B\subset\T$, the condition $\dist(\zeta,B)\to0$ also forces
$|\zeta|\to1$. Equations \eqref{eq:A-blowup} and \eqref{eq:Q} give
\begin{equation}\label{eq:Q-blowup}
 |Q(\zeta)|\longrightarrow+\infty
 \quad\hbox{uniformly as}\quad
 \dist(\zeta,B)\longrightarrow0,
 \qquad \zeta\in\D.
\end{equation}

Let
\begin{equation}\label{eq:Phi-g}
 \Phi=Q\circ\chi^{-1}\colon D\longrightarrow\Omega,
 \qquad
 g=\Phi^{-1}=\chi\circ Q^{-1}.
\end{equation}

Choose $0<\eps<1/4$ so small that
$\eps C_0a^m<h_{m-1}/4$ for every $m\geq2$, and put
\begin{equation}\label{eq:closed-discs}
 \rho_\theta(t)=\eps d_D(te^{i\theta}),
 \qquad
 T_{\theta,t}=\overline{\Delta(te^{i\theta},\rho_\theta(t))}
\end{equation}
for $\theta\in K$ and $0\leq t<1$. These discs are compact subsets of $D$.
Since $\rho_\theta(0)=\eps d_D(0)$, the disc $T_{\theta,0}$ is the same
for every $\theta\in K$.

Suppose that $r_m\leq t<r_{m+1}$ and $m\geq2$. Put
$v=v_{m-1}(\theta)$. We have
\begin{equation}\label{eq:disc-crosscut-separation}
 r_m-y_{m-1}=\frac{h_{m-1}}4,
 \qquad
 \rho_\theta(t)\leq\eps C_0a^m,
 \qquad
 \frac{a^m}{h_{m-1}}
 =\frac{2a}{1-b}\left(\frac ab\right)^{m-1}.
\end{equation}

One fixed choice of $\eps$ makes the closed disc $T_{\theta,t}$ lie
beyond $\gamma_v$. The disc is connected and misses $\gamma_v$. It
follows that
\begin{equation}\label{eq:disc-in-Vv}
 \chi^{-1}(T_{\theta,t})\subset V_v.
\end{equation}

Equations \eqref{eq:A-on-Vv}, \eqref{eq:boundary-distance}, and
\eqref{eq:disc-in-Vv} imply
\begin{equation}\label{eq:A-disc-lower}
 \inf_{\xi\in T_{\theta,t}}
 \Real{A(\chi^{-1}(\xi))}
 \geq c_2\left(\frac{qb}{a}\right)^m-C_2
\end{equation}
and
\begin{equation}\label{eq:rho-upper-log}
 \log\frac2{\rho_\theta(t)}
 \leq C_3+(m+1)\log\frac1a.
\end{equation}

The first lower bound grows exponentially in $m$, while the right-hand side
of \eqref{eq:rho-upper-log} grows linearly. For every $t_0<1$, the discs
$T_{\theta,t}$ with $\theta\in K$ and $0\leq t\leq t_0$ lie in one
compact subset of $D$, and their radii have a positive lower bound.
There is consequently a real constant $C_*$ such that
\begin{equation}\label{eq:margin}
 \Real{A(\chi^{-1}(\xi))}+C_*
 \geq\log\frac2{\rho_\theta(t)}
\end{equation}
for every $\theta\in K$, $0\leq t<1$, and
$\xi\in T_{\theta,t}$.

The images of these discs tend uniformly to infinity as $t\to1^-$. Indeed,
if
$t_j\to1$, $\theta_j\in K$, and
$\xi_j\in T_{\theta_j,t_j}$, then \eqref{eq:disc-in-Vv} and
\eqref{eq:delta-n} imply
\[
 \dist(\chi^{-1}(\xi_j),B)\longrightarrow0.
\]

Equation \eqref{eq:Q-blowup} gives
\begin{equation}\label{eq:uniform-disc-escape}
 \inf_{\theta\in K}\inf_{\xi\in T_{\theta,t}}
 |\Phi(\xi)|\longrightarrow+\infty
 \quad (t\to1^-).
\end{equation}

We next find a closed set $F\subset\Omega$ that contains every image
$\Phi(T_{\theta,t})$ and on which continuous functions that are holomorphic
in the interior can be approximated uniformly by entire functions.
The geometry of the starlike domain $\Omega$ will ensure that its
complement has the required connectivity properties.

We use Arakelian's approximation theorem in the following classical form
\cite{Arakelian1964}.

\begin{theorem}[Arakelian]\label{thm:arakelian}
Let $F$ be a closed subset of $\C$, and let $\Chat$ denote the Riemann
sphere. The following conditions are equivalent.
\begin{enumerate}
\item The set $\Chat\setminus F$ is connected and locally connected at
infinity.
\item If $u$ is continuous on $F$ and holomorphic in $F^\circ$, then, for
every $\eta>0$, there is an entire function $U$ such that
\[
 |U-u|<\eta
 \quad\hbox{on }F.
\]
\end{enumerate}
A closed set that satisfies these conditions is called an Arakelian set.
\end{theorem}

\begin{lemma}\label{lem:approximation-set}
Let $\Omega\subsetneq\C$ be a domain starlike with respect to the origin.
For every set $S\subset\Omega$ closed in $\C$, there is an Arakelian set
$F$ such that
\[
 S\subseteq F\subset\Omega.
\]
\end{lemma}

\begin{proof}
If $S=\varnothing$, take $F=\{0\}$. Henceforth suppose that
$S\ne\varnothing$.

For $\alpha\in[0,2\pi)$, put
\[
 R_\alpha=\sup\{r>0:re^{i\alpha}\in\Omega\},
 \qquad
 \Theta=\{\alpha\in[0,2\pi):R_\alpha<+\infty\}.
\]
Since $\Omega$ is open and starlike with respect to the origin,
$R_\alpha>0$ and
\begin{equation}\label{eq:slit-complement}
 \C\setminus\Omega=\bigcup_{\alpha\in\Theta}L_\alpha,
 \qquad
 L_\alpha=\{re^{i\alpha}:r\geq R_\alpha\}.
\end{equation}

Fix $\alpha\in\Theta$. Since $S$ is closed and
$R_\alpha e^{i\alpha}\notin S$, choose $0<\eta_\alpha<R_\alpha/2$ so that
\[
 \{re^{i\alpha}:R_\alpha-\eta_\alpha<r\leq R_\alpha\}\cap S=\varnothing.
\]

Put $r_\alpha=R_\alpha-\eta_\alpha$. For $r>r_\alpha$, define
\[
 d_\alpha(r)=\dist(re^{i\alpha},S),
 \qquad
 \omega_\alpha(r)=\min\left\{\frac18,\frac{d_\alpha(r)}{4r}\right\}
\]
and
\begin{equation}\label{eq:open-neighborhood}
 O_\alpha=\{re^{i(\alpha+u)}:r>r_\alpha,\ |u|<\omega_\alpha(r)\}.
\end{equation}

The function $d_\alpha$ is positive and continuous. Hence $O_\alpha$ is open.
It contains $L_\alpha$ and is disjoint from $S$. To see that $O_\alpha$ is
disjoint from $S$, note that
\[
 |re^{i(\alpha+u)}-re^{i\alpha}|
 \leq r|u|<\frac14d_\alpha(r).
\]

Let $O=\bigcup_{\alpha\in\Theta}O_\alpha$ and $F=\C\setminus O$.
Then $F$ is closed, $S\subseteq F$, and \eqref{eq:slit-complement} gives
$F\subset\Omega$.

Every point $z=re^{i(\alpha+u)}\in O_\alpha$ can be joined to infinity in
$O_\alpha$ without decreasing its modulus. One such curve is
\[
 \gamma_z(t)=t\exp\left(i\left[\alpha+
 u\frac{\omega_\alpha(t)}{\omega_\alpha(r)}\right]\right),
 \qquad t\geq r.
\]

Indeed, $|u|\omega_\alpha(t)/\omega_\alpha(r)<\omega_\alpha(t)$ for every
$t\geq r$. Thus every component of $O$ is unbounded. The same is true
of every component of $O\setminus\overline{\Delta(0,R)}$ for every $R>0$.
It follows that $O\cup\{\infty\}=\Chat\setminus F$ is connected and that
\[
 \{\infty\}\cup
 \bigl(O\setminus\overline{\Delta(0,R)}\bigr),
 \qquad R>0,
\]

is a connected relative neighborhood basis at infinity. Therefore
$\Chat\setminus F$ is locally connected at infinity. By Arakelian's
criterion, $F$ is an Arakelian set.
\end{proof}

Define
\begin{equation}\label{eq:S}
 S=\bigcup_{\theta\in K}\bigcup_{0\leq t<1}
 \Phi(T_{\theta,t}).
\end{equation}

This set is closed in $\C$. To prove this, consider a convergent
sequence of its points. Formula \eqref{eq:uniform-disc-escape} prevents the
corresponding parameters $t_j$ from tending to one. After passing to a
subsequence, $\theta_j\to\theta\in K$ and $t_j\to t<1$. The distance
function $d_D$ is continuous. The preimages under $\Phi$ then lie in one compact
subset of $D$, and a further subsequence converges to a point of
$T_{\theta,t}$. Continuity of $\Phi$ proves the claim. We apply
\lemref{lem:approximation-set} and fix an Arakelian set $F$ such that
\begin{equation}\label{eq:F}
 S\subseteq F\subset\Omega.
\end{equation}
This set will be used for approximation by entire functions in the next
section.

\section{Proof of Theorem~\ref{thm:main}}\label{sec:entire-function}

\begin{proof}[Proof of \thmref{thm:main}]
We use two uniform approximations to obtain an error bound that decreases
with $\rho_\theta(t)$. On $\Omega$, put
\begin{equation}\label{eq:P}
 P=A\circ Q^{-1}+C_*.
\end{equation}

By \eqref{eq:margin},
\begin{equation}\label{eq:exp-margin}
 e^{-\Real{P(w)}}\leq\frac12\rho_\theta(t)
 \quad\hbox{for }w\in\Phi(T_{\theta,t}).
\end{equation}

Apply \thmref{thm:arakelian} first to $-P|_F$. There is an entire
function $H$ such that
\begin{equation}\label{eq:H-approximation}
 |H+P|<\log2
 \quad\hbox{on }F.
\end{equation}

The function $ge^{-H}$ is holomorphic on $\Omega$. A second application of
\thmref{thm:arakelian} gives an entire function $G$ such that
\begin{equation}\label{eq:G-approximation}
 |G-ge^{-H}|<\frac18
 \quad\hbox{on }F.
\end{equation}

Define
\begin{equation}\label{eq:f}
 f=e^HG.
\end{equation}

Equations \eqref{eq:H-approximation} and \eqref{eq:G-approximation} imply
\begin{equation}\label{eq:f-g-error}
 |f-g|
 <\frac14e^{-\Real{P}}
 \leq\frac18\rho_\theta(t)
 \quad\hbox{on }\Phi(T_{\theta,t}).
\end{equation}

Fix $\theta\in K$ and $0\leq t<1$. Let
\[
 \mathcal U_{\theta,t}
 =\Phi\bigl(\Delta(te^{i\theta},\rho_\theta(t))\bigr).
\]

This is a bounded Jordan domain because $\Phi$ is conformal on a
neighborhood of the closed disc $T_{\theta,t}$. On its boundary,
\[
 |g(w)-te^{i\theta}|=\rho_\theta(t).
\]

The strict inequality \eqref{eq:f-g-error} and Rouch\'e's theorem show that
$f(w)-te^{i\theta}$ and $g(w)-te^{i\theta}$ have the same number of zeros
in $\mathcal U_{\theta,t}$. The latter function has exactly one zero.
Therefore there is exactly one point $z_\theta(t)\in\mathcal U_{\theta,t}$
such that
\begin{equation}\label{eq:inverse-lift}
 f(z_\theta(t))=te^{i\theta}.
\end{equation}

The zero is simple because it is the only zero counted with multiplicity.

At $t=0$, the disc $T_{\theta,0}$ and the domain
$\mathcal U_{\theta,0}$ do not depend on $\theta$. Hence the points
$z_\theta(0)$ all coincide with one simple zero $z_0$ of $f$. Denote by
$\phi$ the inverse germ of $f$ at zero for which $\phi(0)=z_0$.

We now show that the points in \eqref{eq:inverse-lift} define the analytic
continuation of $\phi$. Fix $t_0<1$. The centres and radii of the discs
$T_{\theta,t}$ depend continuously on $t$, and these discs lie in one
compact subset of $D$ for $t$ near $t_0$. Any sequence $t_j\to t_0$ has a
subsequence for which $z_\theta(t_j)$ converges. Its limit is a zero of
$f(w)-t_0e^{i\theta}$ in
$\overline{\mathcal U_{\theta,t_0}}$. The strict boundary inequality
places it in the interior, and uniqueness makes it
$z_\theta(t_0)$. Thus $z_\theta(t)$ is continuous. Since every zero is
simple, the inverse function theorem makes it locally holomorphic in the
range variable. Uniqueness on overlaps patches these local inverses to the
analytic continuation of $\phi$ along $[0,e^{i\theta})$.

Finally, $z_\theta(t)$ belongs to the image
$\Phi(T_{\theta,t})$. Formula \eqref{eq:uniform-disc-escape} gives
\begin{equation}\label{eq:inverse-escape}
 |z_\theta(t)|\longrightarrow+\infty
 \quad (t\to1^-).
\end{equation}

The values in \eqref{eq:inverse-lift} remain bounded while the moduli of their
preimages tend to infinity. Hence $f$ cannot be a polynomial and is transcendental.
The branch cannot continue holomorphically through $e^{i\theta}$. If
continuation of inverse branches is allowed to be meromorphic, a pole there
is impossible as well. Composition of the transcendental entire function
$f$ with a pole has an essential singularity and cannot agree with the
identity. Thus
\[
 \{e^{i\theta}:\theta\in K\}\subseteq E(f,\phi).
\]

The left-hand set has positive $s$-dimensional Hausdorff measure and
positive logarithmic capacity by
\lemref{lem:positive-capacity}. This proves \thmref{thm:main}.
\end{proof}

\begin{remark}
The entire function $f$ constructed above has infinite lower order:
\[
 \underline{\rho}(f)
 =\liminf_{r\to\infty}\frac{\log\log M(r,f)}{\log r}=+\infty,
 \qquad M(r,f)=\max_{|z|=r}|f(z)|.
\]
Indeed, \eqref{eq:inverse-lift} and \eqref{eq:inverse-escape} show that every
point of $\{e^{i\theta}:\theta\in K\}$ is a finite asymptotic value of $f$.
By the Denjoy--Carleman--Ahlfors theorem, a transcendental entire function
of finite lower order $\mu$ has at most $2\mu$ finite asymptotic values
\cite[p.~118]{Langley2019}.
\end{remark}

\begingroup
\setlength{\emergencystretch}{2em}
\renewcommand{\refname}
\endgroup

\end{document}